\documentclass[12pt]{article}
\usepackage{tikz}
\usetikzlibrary{arrows}
\usepackage[framemethod=tikz]{mdframed}
\usepackage{wrapfig}
\usepackage{amsmath,amssymb}
\usepackage{type1cm}
\usepackage{amsthm}
\usepackage{mathrsfs}
\usepackage{mathtools}
\usepackage{enumerate}
\usepackage[all]{xy} 
\usepackage[vcentermath,enableskew]{youngtab}
\usepackage{ytableau}
\usepackage{diagbox}
\usepackage{longtable}
\AtBeginDocument{%
   \def\MR#1{}
}

\DeclareMathOperator{\GSp}{GSp}

\DeclareMathOperator{\diag}{diag}

\DeclareMathOperator{\Adm}{Adm}

\DeclareMathOperator{\pr}{pr}

\DeclareMathOperator{\supp}{supp}

\DeclareMathOperator{\LP}{LP}

\newcommand\F{\mathbb{F}}

\newcommand\cA{\mathscr A}

\newcommand\cC{C}

\newcommand\Q{\mathbb Q}
\newcommand\bQp{\breve{\Q}_p}
\newcommand\R{\mathbb R}

\newcommand\Z{\mathbb Z}
\newcommand\tW{\tilde W}
\newcommand\tw{\tilde w}

\newcommand\bS{\mathbb S}
\newcommand\SW{{^{\bS}\tilde W}}
\newcommand\SAdm{{^{\bS}\mathrm{Adm}}}

\newcommand\tS{\tilde{\mathbb S}}

\newcommand\vph{\varphi}

\newcommand\la{\langle}
\newcommand\ra{\rangle}

\theoremstyle{definition}
\newtheorem{theo}{Theorem}[section]
\newtheorem{prop}[theo]{Proposition}

\newtheorem{lemm}[theo]{Lemma}

\newtheorem{thm}{Theorem}[section]

\begin{document}
\title{The number of Ekedahl-Oort strata intersecting the supersingular locus in the Siegel modular variety}
\author{Ryosuke Shimada}
\date{}
\maketitle

\begin{abstract}
Let $N_g$ be the number of Ekedahl-Oort strata intersecting the supersingular locus in the Siegel modular variety.
We prove that $N_g\sim 2^{g-1}$ as $g\to\infty$.
\end{abstract}

\section{Introduction}
\label{introduction}
Let $p$ be a prime and let $\cA_g$ denote the moduli space of principally polarized abelian varieties of dimension $g$ over $\overline{\F}_p$.
Let $\mathscr S_g\subset\cA_g$ denote the supersingular locus.
The {\it Ekedahl-Oort stratification} of $\cA_g$ is defined by the isomorphism class of the $p$-torsion $A[p]$.
By \cite[Theorem 9.4]{Oort01}, these strata are indexed by {\it elementary sequences}
$$\vph=(\vph(1),\ldots,\vph(g)),\qquad \vph(0)=0,\quad \vph(i)-\vph(i-1)\in\{0,1\}.$$
Thus there are $2^g$ Ekedahl-Oort strata in $\cA_g$.
We denote the stratum corresponding to $\vph$ by $EO_\vph$ and set
$$N_g=\sharp\{\vph\mid EO_\vph\cap\mathscr S_g\neq\emptyset\}.$$

By \cite[Proposition 2.5(1)]{GC12}, we have $EO_\vph\subseteq\mathscr S_g$ if and only if $\vph\left(\left\lceil\frac{g}{2}\right\rceil\right)=0$.
Therefore,
$$\sharp\{\vph\mid EO_\vph\subseteq\mathscr S_g\}=2^{\left\lfloor\frac{g}{2}\right\rfloor}.$$
For $g=3$ or $4$, the strata intersecting $\mathscr S_g$ were determined in \cite{ST24} using affine Deligne-Lusztig varieties.
In this paper, we study $N_g$ as $g$ tends to infinity.

For an abelian variety $A$ over $\overline{\F}_p$, its {\it $p$-rank} $f(A)$ is defined by $A[p](\overline{\F}_p)\cong(\Z/p\Z)^{f(A)}$.
Since the $p$-rank depends only on $A[p]$, it is constant on each Ekedahl-Oort stratum, with value $\max\{i\mid\vph(i)=i\}$ on $EO_\vph$.
Hence it is zero if and only if $\vph(1)=0$, and there are $2^{g-1}$ such strata.
Since every supersingular abelian variety has $p$-rank $0$, $EO_\vph\cap\mathscr S_g\neq\emptyset$ implies $\vph(1)=0$.
Therefore $N_g\leq 2^{g-1}$.
Our main result says that asymptotically almost all of these strata intersect $\mathscr S_g$.
\begin{thm}[Theorem \ref{asymptotic}]
\label{main thm}
We have
$$N_g\sim 2^{g-1}\quad\text{and}\quad \lim_{g\to\infty}\frac{N_{g+1}}{N_g}=2.$$
\end{thm}

The strategy of the proof is as follows:
We first use affine Deligne-Lusztig varieties to obtain a useful criterion for $EO_\vph\cap\mathscr S_g\neq\emptyset$ (cf.\ Proposition \ref{cone criterion}).
We then use this criterion to estimate $N_g$.
This criterion also shows that $N_g<2^{g-1}$ for $g\geq3$.

The author shared a preliminary version of this paper with Naotake Takao and Shushi Harashita on September 6, 2026, at 21:50 PDT.
The author then submitted the paper for publication on September 8, 2026, at 21:31 PDT, without first uploading it to arXiv.
The author later became aware of the closely related work of Muller \cite{Muller26}, which appeared on arXiv after this submission.

\textbf{Acknowledgments:}
The author would like to thank Shushi Harashita and Naotake Takao for helpful comments.

The author was supported by JSPS KAKENHI Grant number JP25K23334.
The author worked on this paper during a stay at UC Berkeley supported by the JSPS
Overseas Research Fellowship. The author would like to thank the university and
his host, Sug Woo Shin, for their hospitality.

OpenAI’s ChatGPT was used to assist with mathematical discussion, language,
and bibliographic searches.

\section{Preliminaries}
\label{preliminaries}
\subsection{Notation}
\label{notation}
Let $\bQp$ be the completion of the maximal unramified extension of $\Q_p$, let $\breve{\Z}_p$ be its ring of integers, and let $\sigma$ be the Frobenius automorphism of $\bQp/\Q_p$.
Set $G=\GSp_{2g}$, and let $T\subset B\subset G$ be the diagonal torus and the upper triangular Borel subgroup, with the symplectic form as in \cite[\S2.3]{GC10}.
Then
$$X_*(T)=\{(m_1,\ldots,m_{2g})\in\Z^{2g}\mid m_1+m_{2g}=m_2+m_{2g-1}=\cdots=m_g+m_{g+1}\}.$$
For $\lambda=(m_1,\ldots,m_{2g})\in X_*(T)$, we write $p^\lambda=\diag(p^{m_1},\ldots,p^{m_{2g}})$.
Set $\mu=(1^{(g)},0^{(g)})\in X_*(T)$, where $1$ and $0$ are repeated $g$ times.
Set
$$\tau=\begin{pmatrix}0&p\cdot 1_g\\1_g&0\end{pmatrix}\in G(\Q_p),$$
where $1_g$ is the identity matrix of size $g$.

Let $W_0$ denote the finite Weyl group and let $\tW=N_G(T)(\bQp)/T(\breve{\Z}_p)$ be the Iwahori-Weyl group.
We also regard $\tau$ as an element of $\tW$.
We identify
$$W_0=\{w\in\mathfrak S_{2g}\mid w(i)+w(2g+1-i)=2g+1\text{ for all }i\}.$$
We have $\tW=X_*(T)\rtimes W_0$ and denote the projection to $W_0$ by $\pr$.
Set
$$s_i=(i\ i+1)(2g-i\ 2g+1-i)\ (1\leq i<g),\qquad s_g=(g\ g+1).$$
Then $\bS=\{s_1,\ldots,s_g\}$ is the set of simple reflections of $W_0$.
The affine Weyl group $W_a$ has simple reflections $\tS=\bS\cup\{s_0\}$, where
$$s_0=p^{(1,0^{(2g-2)},-1)}(1\ 2g).$$
We have $\tW=W_a\rtimes\Omega$, where $\Omega=\langle\tau\rangle\cong\Z$ is the subgroup of length $0$ elements.
We write $\ell$ and $\leq$ for the length function and the Bruhat order on $\tW$.
For $J\subseteq\tS$, let $W_J$ be the subgroup generated by $J$, and let $^J\tW$ be the set of minimal length representatives for $W_J\backslash\tW$.

The $\sigma$-conjugacy class of $\tau$ is basic, and
$$\tau s_i\tau^{-1}=s_{g-i}\qquad(0\leq i\leq g).$$
For $u\in W_a$, let $\supp(u)$ denote the simple reflections occurring in a reduced expression of $u$.
We define $\supp_\sigma(u\tau)$ to be the smallest $\tau$-stable subset of $\tS$ containing $\supp(u)$.

Let $\Phi=\Phi(G,T)$ denote the set of roots of $T$ in $G$.
We denote by $\Phi_+$ the set of positive roots distinguished by $B$.
We use the reflection representation
$$V=X_*(T)_\R/X_*(Z(G))_\R\cong\R^g.$$
On $V$, the image of $\mu$ is $(\frac{1}{2},\ldots,\frac{1}{2})$, the reflections $s_i$ with $i<g$ interchange the $i$-th and $(i+1)$-th coordinates, and $s_g$ changes the sign of the last coordinate.
Writing $e_1,\ldots,e_g$ for the standard basis and identifying $V$ with its dual, we have
$$\Phi_+=\{e_i-e_j,e_i+e_j\mid 1\leq i<j\leq g\}\cup\{2e_i\mid 1\leq i\leq g\}.$$
The closed dominant chamber is $\overline{C^+}=\{x=(x_1,\ldots,x_g)\in V\mid x_1\geq\cdots\geq x_g\geq0\}$.

\subsection{Affine Deligne-Lusztig Varieties}
\label{ADLV}
Set $K=G(\breve{\Z}_p)$, and let $I\subset K$ be the inverse image of the lower triangular Borel subgroup under reduction modulo $p$.
For $w\in\tW$, we have the {\it affine Deligne-Lusztig varieties}
\begin{align*}
X_w(\tau)&=\{xI\in G(\bQp)/I\mid x^{-1}\tau\sigma(x)\in IwI\},\\
X_\mu(\tau)&=\{xK\in G(\bQp)/K\mid x^{-1}\tau\sigma(x)\in Kp^\mu K\}.
\end{align*}
Set
$$\Adm(\mu)=\{w\in\tW\mid w\leq p^{v\mu}\text{ for some }v\in W_0\},\qquad
\SAdm(\mu)=\Adm(\mu)\cap\SW.$$
Since $\mu$ is minuscule, the Ekedahl-Oort stratification is
$$X_\mu(\tau)=\bigsqcup_{w\in\SAdm(\mu)}\pi(X_w(\tau)),$$
where $\pi\colon G(\bQp)/I\to G(\bQp)/K$ is the projection (cf.\ \cite[Theorem 3.2.1]{GH15}).

In the Siegel case, there are three parametrizations of the Ekedahl-Oort strata (cf.\ \cite[\S2.4]{GC12}).
Set $J=\bS\setminus\{s_g\}$.
Let $^J W_0$ denote the set of minimal length representatives for $W_J\backslash W_0$.
There is a bijection $\vph\mapsto w_\vph$ from elementary sequences to $^J W_0$ characterized by
$$\vph(i)=\sharp\{1\leq j\leq i\mid w_\vph(j)>g\}\qquad(1\leq i\leq g).$$

The bijection $^J W_0\xrightarrow{\sim}\SAdm(\mu)$ is given by
$$w_\vph\longmapsto\tw_\vph\coloneqq\tau w_\vph.$$
By \cite[Lemma 7.6]{GHN19}, we have
$$EO_\vph\cap\mathscr S_g\neq\emptyset\quad\Longleftrightarrow\quad X_{\tw_\vph}(\tau)\neq\emptyset.$$

Let $\delta^+\colon\Phi\to\{0,1\}$ be the indicator function of $\Phi_+$.
For $w=p^\mu y\in\SAdm(\mu)$, define
$$\LP(w)=\{v\in W_0\mid \la v\alpha,y^{-1}\mu\ra+\delta^+(v\alpha)-\delta^+(yv\alpha)\geq0
\text{ for all }\alpha\in\Phi_+\}.$$
The following non-emptiness criterion was proved by Schremmer in \cite[Proposition 6]{Schremmer23}; see also \cite[Theorem 2.5]{ST24}.
\begin{theo}
\label{empty}
Let $w\in\SAdm(\mu)$.
Then $X_w(\tau)\neq\emptyset$ if and only if at least one of the following conditions holds:
\begin{enumerate}[(i)]
\item $W_{\supp_\sigma(w)}$ is finite.
\item $\supp(v^{-1}\pr(w)v)=\bS$ for every $v\in\LP(w)$.
\end{enumerate}
\end{theo}

For $w=\tw_\vph$, condition (i) is equivalent to $EO_\vph\subseteq\mathscr S_g$ (cf.\ \cite[Proposition 5.6]{GHN19}).

\section{Proof of the Main Theorem}
\label{criterion}
\subsection{A Non-emptiness Criterion}
\label{LP}
We identify the set $\{0,1\}^g$ with the set of elementary sequences by setting $\vph(i)=\sum_{j=1}^i\nu(j)$ for $\nu=(\nu(1),\ldots,\nu(g))\in\{0,1\}^g$.
We write
$$w_\nu=w_\vph,\qquad\tw_\nu=\tw_\vph=\tau w_\nu,\qquad EO_\nu=EO_\vph.$$
Set $y_\nu=\pr(\tw_\nu)=(1\ g+1)\cdots(g\ 2g)w_\nu\in W_0$.
The definition of $w_\vph$ in \S\ref{ADLV} gives
\begin{align}
\label{signed permutation}
y_\nu(e_i)&=
\begin{cases}
e_{\vph(i)}&\text{if }\nu(i)=1,\\
-e_{g-i+\vph(i)+1}&\text{if }\nu(i)=0.
\end{cases}
\end{align}
Define
$$
\cC_{\nu}=\left\{x\in\R^g\ \middle|\
\begin{aligned}
&y_\nu x=x,\quad x_i\geq0\ (\nu(i)=1),\\
&x_i\geq x_j\ (i<j,\ \nu(i)=1,\ \nu(j)=0)
\end{aligned}\right\}.
$$
\begin{prop}
\label{cone criterion}
We have
$$EO_\nu\cap\mathscr S_g\neq\emptyset
\quad\Longleftrightarrow\quad
\vph\left(\left\lceil\frac{g}{2}\right\rceil\right)=0\ \text{ or }\ \cC_{\nu}=\{0\}.$$
\end{prop}
\begin{proof}
For $\alpha\in\Phi$, set
$$L_{\nu}(\alpha)=\la\alpha,y_\nu^{-1}\mu\ra+\delta^+(\alpha)-\delta^+(y_\nu\alpha).$$
Since $y_\nu^{-1}\mu=(\nu(1)-\frac{1}{2},\ldots,\nu(g)-\frac{1}{2})$ on $V$, we have, for $i<j$,
\begin{align*}
L_{\nu}(e_i-e_j)&=\nu(i)(1-\nu(j)),&
L_{\nu}(e_i+e_j)&=\nu(i)\nu(j),&
L_{\nu}(2e_i)&=\nu(i).
\end{align*}
Also, $L_{\nu}(-\alpha)=-L_{\nu}(\alpha)$.
Set
\begin{align*}
D&=\{x\in V\mid \la\alpha,x\ra\geq0\text{ for all }\alpha\in\Phi\text{ with }L_\nu(\alpha)>0\}\\
&=\{(x_i)_i\in\R^g\mid x_i\geq0\text{ if }\nu(i)=1,\quad
x_i\geq x_j\text{ if }i<j,\ \nu(i)=1,\ \nu(j)=0\}.
\end{align*}
For $v\in W_0$, the definition of $\LP$ and the identity $L_\nu(-\alpha)=-L_\nu(\alpha)$ give
\begin{align*}
v\in\LP(\tw_\nu)
\Longleftrightarrow L_\nu(\alpha)\geq0\text{ for all }\alpha\in v\Phi_+
\Longleftrightarrow\{\alpha\in\Phi\mid L_\nu(\alpha)>0\}\subseteq v\Phi_+.
\end{align*}
For a root $\alpha$, we have $\alpha\in v\Phi_+$ if and only if $\la\alpha,x\ra\geq0$ for every $x\in v\overline{C^+}$.
Consequently,
$$v\in\LP(\tw_\nu)\quad\Longleftrightarrow\quad v\overline{C^+}\subseteq D.$$
Hence $\bigcup_{v\in\LP(\tw_\nu)}v\overline{C^+}\subseteq D$.
Conversely, let $x\in D$ and choose $v\in W_0$ of minimal length among those satisfying $x\in v\overline{C^+}$.
If $v\notin\LP(\tw_\nu)$, the formulas for $L_\nu$ and the criterion above give a root $\alpha\in\Phi_+$ with $L_\nu(\alpha)>0$ and $v^{-1}\alpha\in-\Phi_+$.
Since $x\in D$ and $v^{-1}x\in\overline{C^+}$, we have
$$0\leq\la\alpha,x\ra=\la v^{-1}\alpha,v^{-1}x\ra\leq0.$$
Thus $s_\alpha x=x$, where $s_\alpha$ denotes the reflection associated with $\alpha$.
Consequently, $x\in s_\alpha v\overline{C^+}$, whereas $\ell(s_\alpha v)<\ell(v)$ since $v^{-1}\alpha\in-\Phi_+$, contradicting the minimality of $\ell(v)$.
Hence $v\in\LP(\tw_\nu)$, so $x$ belongs to the union above.
Therefore,
$$D=\bigcup_{v\in\LP(\tw_\nu)}v\overline{C^+}.$$
By the definition of $\cC_\nu$, we have $\cC_{\nu}=\{x\in D\mid y_\nu x=x\}$.

For $u\in W_0$, we have $\supp(u)\subsetneq\bS$ if and only if $u$ fixes a nonzero point of $\overline{C^+}$, since the stabilizers of such points are precisely the proper standard parabolic subgroups of $W_0$.
It follows that
$$\supp(v^{-1}y_\nu v)=\bS\text{ for every }v\in\LP(\tw_\nu)\quad\Longleftrightarrow\quad\cC_{\nu}=\{0\}.$$
By \cite[Proposition 2.5(1)]{GC12}, condition (i) of Theorem \ref{empty} is equivalent to $\vph\left(\left\lceil\frac{g}{2}\right\rceil\right)=0$.
Thus the assertion follows from Theorem \ref{empty}.
\end{proof}

This criterion is much more efficient than checking all length positive elements.
Set
$$E_g=\sharp\{\nu\in\{0,1\}^g\mid \nu(1)=0,\ \cC_{\nu}\neq\{0\}\}.$$
By Proposition \ref{cone criterion}, we have
\begin{align}
\label{error bound}
0&\leq2^{g-1}-N_g\leq E_g.
\end{align}
For $g\geq3$, the type $\nu=(0,1,0,\ldots,0)$ satisfies $\vph\left(\left\lceil\frac{g}{2}\right\rceil\right)=1$ and $e_1+e_2-e_g\in\cC_\nu\setminus\{0\}$.
Thus $EO_\nu\cap\mathscr S_g=\emptyset$ by Proposition \ref{cone criterion}, and $N_g<2^{g-1}$ since $\nu(1)=0$.

\subsection{Estimating $E_g$}
\label{counting}

For $\nu\in\{0,1\}^g$, let $\pi_\nu\in\mathfrak S_g$ be determined by $y_\nu^{-1}(e_i)=\pm e_{\pi_\nu(i)}$.
Equivalently, we have
$$y_\nu^{-1}\bigl(\{i,2g+1-i\}\bigr)=\{\pi_\nu(i),2g+1-\pi_\nu(i)\}\qquad(1\leq i\leq g).$$
By \eqref{signed permutation}, we have
$$\pi_\nu(\vph(i))=i\quad(\nu(i)=1),\qquad
\pi_\nu(g+1-i+\vph(i))=i\quad(\nu(i)=0).$$
Thus $\pi_\nu(1),\ldots,\pi_\nu(\vph(g))$ list the indices $i$ with $\nu(i)=1$ in increasing order, and $\pi_\nu(\vph(g)+1),\ldots,\pi_\nu(g)$ list those with $\nu(i)=0$ in decreasing order.
Conversely, every permutation increasing on $\{1,\ldots,\vph(g)\}$ and decreasing on $\{\vph(g)+1,\ldots,g\}$ arises from a unique element of $\{0,1\}^g$ having $\vph(g)$ entries equal to $1$.
Color each index $i$ in each cycle of $\pi_\nu$ by $a$ if $i\leq\vph(g)$ and by $b$ if $i>\vph(g)$.
Each colored cycle, considered up to cyclic rotation, is called a {\it necklace}.
By \cite[Theorem 2.2]{Steinhardt10}, the resulting multiset of necklaces
determines $\nu$.

For example, let $\nu=(0,0,1,0,1,0)$, so $g=6$ and $\vph(g)=2$. Then
$$\begin{array}{c|rrrrrr}
i&1&2&3&4&5&6\\\hline
\pi_\nu(i)&3&5&6&4&2&1
\end{array}$$
The cycles of $\pi_\nu$ are $(1\ 3\ 6)$, $(2\ 5)$ and $(4)$, giving the necklaces $(abb)$, $(ab)$ and $(b)$, respectively.
A {\it run} is a maximal nonempty block of consecutive identical symbols; for a necklace, this is understood cyclically.

\begin{lemm}\label{necklace count}
Let $r\geq1$ be an integer.
The number of elements $\nu\in\{0,1\}^g$ for which
$\pi_{\nu}$ has a cycle of length at least $r$ whose colored necklace has only
even runs of $b$ is at most
$$
2^g\sum_{l=r}^{\infty}\left(\frac{1+\sqrt5}{4}\right)^l.
$$
\end{lemm}
\begin{proof}
For $l\geq0$, let $c_l$ denote the number of elements of $\{a,b\}^l$ obtained by concatenating $a$ and $bb$.
Since $c_0=c_1=1$ and $c_l=c_{l-1}+c_{l-2}$ for $l\geq2$, we obtain $c_l\leq\left(\frac{1+\sqrt5}{2}\right)^l$ by induction on $l$.
Every necklace of length $l$ with only even runs of $b$ can be read as a concatenation of $a$ and $bb$, so there are at most $c_l$ such necklaces.

Fix $r\leq l\leq g$.
For each $\nu$ for which $\pi_\nu$ has such a cycle of length $l$, choose one such cycle $\gamma$ and form $\nu'\in\{0,1\}^{g-l}$ by deleting the entries $\nu(i)$ with $i$ in $\gamma$.
As noted above, the map sending $\nu$ to the pair consisting of the necklace of $\gamma$ and $\nu'$ is injective.
There are at most $c_l2^{g-l}$ such pairs, so the number of such $\nu$ is at most $$c_l2^{g-l}\leq\left(\frac{1+\sqrt5}{2}\right)^l2^{g-l}=2^g\left(\frac{1+\sqrt5}{4}\right)^l.$$
The assertion follows by summing over $l\geq r$.
\end{proof}

\begin{lemm}\label{long cycle}
Suppose that $\nu(1)=0$, that $\cC_{\nu}\neq\{0\}$, and that every run of $0$ in
$\nu$ has length at most $M<g$. Then $\pi_{\nu}$ has a cycle of length at least
$$
\frac{\log(g-M+2)}{\log(M+1)}
$$
whose colored necklace has only even runs of $b$.
\end{lemm}
\begin{proof}
Choose $0\neq x\in\cC_\nu$. Since $\nu(\pi_\nu(i))=1$ for $i\leq\vph(g)$, \eqref{signed permutation} and the definition of $\cC_\nu$ give
$$
x_{\pi_\nu(i)}=
\begin{cases}
x_i&i\leq\vph(g),\\
-x_i&i>\vph(g),
\end{cases}
\qquad x_i\geq0\quad(i\leq\vph(g)).
$$
Some coordinate is positive, since $x_i<0$ implies $i>\vph(g)$ and $x_{\pi_\nu(i)}=-x_i>0$.
Let $j$ be the smallest index with $x_j>0$.
Since $\nu(1)=0$, we have $\vph(j)<j$.
If $\nu(j)=1$, then $y_\nu x=x$ implies $x_{\vph(j)}=x_j>0$, a contradiction.
Hence $\nu(j)=0$.
If $i<j$ and $\nu(i)=1$, then $\nu(j)=0$ implies $x_i\geq x_j>0$, a contradiction.
Thus $\nu(1)=\cdots=\nu(j)=0$, so $j\leq M$ and $\pi_\nu(g-j+1)=j$.

Let $l$ be the length of the cycle of $\pi_\nu$ containing $j$.
For each index $i$ in this cycle, we have $x_i=\pm x_j$, and $x_i=x_j>0$ if $i\leq\vph(g)$.
If its necklace is of the form
$$(\cdots a\,\underbrace{b\cdots b}_{t}\,a\cdots)
\quad\text{or}\quad
(a\,\underbrace{b\cdots b}_{t}),$$
then $x_{\pi_\nu(i)}=-x_i$ for $i>\vph(g)$ gives $x_j=(-1)^t x_j$, so $t$ is even.
If every index $i$ in this cycle satisfies $i>\vph(g)$, then $x_j=(-1)^l x_j$, hence $l$ is even.

For $i\leq\vph(g)$, at most $i$ runs of $0$ precede $\nu(\pi_\nu(i))$, and each run has length at most $M$.
Hence $\pi_\nu(i)\leq (M+1)i$.
For $i>\vph(g)$, we have $\pi_\nu(i)\leq g\leq\vph(g)+M(\vph(g)+1)<(M+1)i$, since $\nu$ has at most $\vph(g)+1$ runs of $0$.
We substitute $i=j,\pi_\nu(j),\ldots,\pi_\nu^{l-2}(j)$ successively into $\pi_\nu(i)\leq(M+1)i$ and use $\pi_\nu^{l-1}(j)=g-j+1$ and $j\leq M$ to obtain
$$
g-M+2\leq\pi_\nu^{l-1}(j)+1
\leq(M+1)^{l-1}j+1\leq(M+1)^l,
$$
which proves the assertion.
\end{proof}

\begin{theo}
\label{asymptotic}
We have
$$N_g\sim 2^{g-1}\quad\text{and}\quad\lim_{g\to\infty}\frac{N_{g+1}}{N_g}=2.$$
\end{theo}
\begin{proof}
Set $M=\left\lceil3\log_2 g\right\rceil$ and
$r=\left\lfloor\frac{\log(g-M+2)}{\log(M+1)}\right\rfloor$.
For $g$ sufficiently large, we have $M<g$ and $r\geq1$.
There are at most $g2^{g-M}$ elements of $\{0,1\}^g$ containing a run of $0$ longer than $M$.
By Lemma \ref{necklace count} and Lemma \ref{long cycle}, we have
$$
\frac{E_g}{2^{g-1}}
\leq g2^{1-M}
+2\sum_{l=r}^{\infty}\left(\frac{1+\sqrt5}{4}\right)^l
\longrightarrow0\quad\text{as }g\to\infty,
$$
since $r\to\infty$ and $\frac{1+\sqrt5}{4}<1$.
Thus the first assertion follows from \eqref{error bound}.
The second assertion follows immediately from the first.
\end{proof}

Table \ref{values} below gives the values of $N_g$ for $2\leq g\leq38$.

\clearpage
\begingroup
\renewcommand{\arraystretch}{1.08}
\setlength{\LTcapwidth}{\textwidth}
\begin{longtable}{r|rrrr}
$g$ & $2^{g-1}$ & $N_g$ & $N_g/N_{g-1}$ & $N_g/2^{g-1}$ \\[5pt]
\hline
\endhead
2 & 2 & 2 & 2.000000 & 100.0000\% \\
3 & 4 & 3 & 1.500000 & 75.0000\% \\
4 & 8 & 6 & 2.000000 & 75.0000\% \\
5 & 16 & 11 & 1.833333 & 68.7500\% \\
6 & 32 & 22 & 2.000000 & 68.7500\% \\
7 & 64 & 44 & 2.000000 & 68.7500\% \\
8 & 128 & 89 & 2.022727 & 69.5312\% \\
9 & 256 & 181 & 2.033708 & 70.7031\% \\
10 & 512 & 370 & 2.044199 & 72.2656\% \\
11 & 1,024 & 753 & 2.035135 & 73.5352\% \\
12 & 2,048 & 1,538 & 2.042497 & 75.0977\% \\
13 & 4,096 & 3,137 & 2.039662 & 76.5869\% \\
14 & 8,192 & 6,381 & 2.034109 & 77.8931\% \\
15 & 16,384 & 12,981 & 2.034321 & 79.2297\% \\
16 & 32,768 & 26,360 & 2.030660 & 80.4443\% \\
17 & 65,536 & 53,454 & 2.027845 & 81.5643\% \\
18 & 131,072 & 108,291 & 2.025873 & 82.6195\% \\
19 & 262,144 & 219,106 & 2.023308 & 83.5823\% \\
20 & 524,288 & 442,858 & 2.021204 & 84.4685\% \\
21 & 1,048,576 & 894,279 & 2.019336 & 85.2851\% \\
22 & 2,097,152 & 1,804,217 & 2.017510 & 86.0318\% \\
23 & 4,194,304 & 3,637,196 & 2.015942 & 86.7175\% \\
24 & 8,388,608 & 7,327,241 & 2.014530 & 87.3475\% \\
25 & 16,777,216 & 14,751,328 & 2.013217 & 87.9248\% \\
26 & 33,554,432 & 29,680,836 & 2.012079 & 88.4558\% \\
27 & 67,108,864 & 59,689,462 & 2.011044 & 88.9442\% \\
28 & 134,217,728 & 119,982,723 & 2.010116 & 89.3941\% \\
29 & 268,435,456 & 241,079,923 & 2.009289 & 89.8093\% \\
30 & 536,870,912 & 484,220,231 & 2.008546 & 90.1930\% \\
31 & 1,073,741,824 & 972,255,881 & 2.007880 & 90.5484\% \\
32 & 2,147,483,648 & 1,951,593,118 & 2.007283 & 90.8781\% \\
33 & 4,294,967,296 & 3,916,352,287 & 2.006746 & 91.1847\% \\
34 & 8,589,934,592 & 7,857,235,927 & 2.006264 & 91.4703\% \\
35 & 17,179,869,184 & 15,760,271,446 & 2.005829 & 91.7369\% \\
36 & 34,359,738,368 & 31,606,224,469 & 2.005437 & 91.9862\% \\
37 & 68,719,476,736 & 63,373,059,588 & 2.005082 & 92.2199\% \\
38 & 137,438,953,472 & 127,047,791,180 & 2.004760 & 92.4394\% \\
\caption{The number of Ekedahl-Oort strata intersecting $\mathscr S_g$.}
\label{values}
\end{longtable}
\endgroup
\clearpage

\bibliographystyle{myamsplain}
\bibliography{reference}
\end{document}